\documentclass{article}
\usepackage[utf8]{inputenc}
\usepackage{amsmath}
\usepackage{amsfonts}
\usepackage[english]{babel}
\usepackage{ragged2e}
\usepackage{amssymb}
\usepackage{comment}
\usepackage{amsthm}

\usepackage{amscd}
\usepackage{enumitem}
\usepackage[margin=1in]{geometry}
\usepackage[dvips]{epsfig}
\usepackage{graphics}
\usepackage{amsmath,amscd}
\usepackage{color}
\usepackage{listings}
\usepackage[english]{babel}
\usepackage{ragged2e}
\usepackage{amssymb}
\newtheorem{theorem}{Theorem}
\newtheorem{definition}{Definition}
\usepackage{enumitem}
\usepackage{fixltx2e}
\usepackage{enumitem}
\usepackage{amsmath,amsthm,amssymb,amsfonts}

\newtheorem{proposition}[theorem]{Proposition}
\newtheorem{conjecture}[theorem]{Conjecture}

\usepackage{mathtools}

\title{Centralizer-like Subgroups Associated with the $n$-Engel Words Inside of Direct Product Groups}
\author{Bridget Lee \and Maggie Reardon \and Faculty Mentor: Dr. Dandrielle Lewis}

\begin{document}

\maketitle

\begin{abstract}
\noindent This research provides a characterization of centralizer-like subgroups associated with the $n$-Engel word in a direct product of groups. Specifically, properties of the set of right $n$-Engel elements inside of direct products are explored. A proof is given to demonstrate the equivalence between the set of right $n$-Engel elements of a direct product of two groups and a direct product of the set of right $n$-Engel elements of each direct factor. This work was inspired by the study of centralizer-like subgroups in \cite{Kappe}. We present additional questions explored during this project, and we propose future research possibilities.
\end{abstract}

\section{Introduction}
This article presents two characterizations of centralizer-like subgroups associated with the $n$-Engel word in $G \times H$, the direct product of groups $G$ and $H$.  These characterizations are provided in the context of centralizer-like subgroups introduced in \cite{Kappe}.\\

\noindent The characterizations, presented and justified in Section $4$, were motivated by the result in Section $3$, Proposition $1$  in Section $2$, \cite{Kappe}, and Proposition $2$ in Section $2$ \cite{McClellan}.  This work is tedious, and the proofs rely heavily on understanding the definitions and background provided in Section $2$.\\

\noindent In Section $5$, we present a conjecture with attempted proofs that would extend our characterizations, if proved successfully in the future.  We include these attempts in this article to provide insight for future interest and suggestions for future projects.  We also provide a detailed example of the direct product of $K_4$ with $S_3$ for which our conjecture holds, where $K_4$ is the Klein four-group and $S_3$ is the symmetric group of order six.  The example is a beautiful presentation of ideas presented in this article along with the understanding of metabelian and solvable groups.\\

\noindent Otherwise the notation is standard.  A standard reference for solvable groups and metabelian groups is \cite{Doerk} and \cite{Huppert}.

\section{Terminology and Background}

Specific definitions and properties regarding groups are introduced in this section to comprehend the mathematical concepts pertaining to centralizer-like subgroups associated with the $n$-Engel word in following sections.  It is understood that the reader has a basic knowledge of abstract algebra, specifically group theory.  We begin with the definition of a commutator. 
\begin{definition}
    Let $G, H$ be a groups and consider the direct product group $G \times H$. Let $(x,y),(g,h) \in G \times H$. A \textbf{commutator} is defined by $$[(x,y)\,,\,(g,h)] = (x,y)^{-1}(g,h)^{-1}(x,y)(g,h).$$
    \textbf{Note:} When $[(x,y)\,,\,(g,h)] = (x,y)^{-1}(g,h)^{-1}(x,y)(g,h) = (1,1)$, $(x,y)$ and $(g,h)$ are said to \textbf{commute}.
\end{definition}
\noindent The definition of a commutator can be used to understand the definition of an $n$-Engel word, but first, a word must be defined.
\begin{definition}
    Let $G \times H$ be a group, a \textbf{word} \cite{McClellan} is a combination of elements in $G \times H$ with the group operation.
\end{definition}
\begin{definition}
    Let $G \times H$ be a group, let $(x,y),(g,h) \in G \times H$, and let $n \in \mathbb{N}$. An \textbf{$n$-Engel word} \cite{McClellan} is defined as $$\varepsilon_n ((x,y),(g,h)) = [(x,y), \text{ }_n (g,h)] = [[(x,y)\,,\,_{n-1} (g,h)]\,,\,(g,h)].$$
\end{definition}
\noindent Thus an $n$-Engel word can be written as a nested commutator. Expanding the commutator of an $n$-Engel word follows the definition of a commutator.
To better understand Definition 1 and Definition 3 we provide an example. \\
\textbf{Example 1.} \textit{Consider the dihedral group of order 8 $$D_8 = \langle r, s \text{ } | \text{ } r^4 = s^2 = 1, r\text{ } s\text{ }r = s\rangle = \{1, r, r^2, r^3, s, rs, r^2s, r^3s\}.$$
A commutator in this group is given by $[r, s] = r^{-1}s^{-1}rs = s$. Now we can evaluate the 3-Engel word of $r$ and $s$ as $\varepsilon_3(r, s) = [[[r, s], s], s] = r^3\text{ }s\text{ }r\text{ }s\text{ }r^3\text{ }s\text{ }r\text{ }s\text{ }r^3\text{ }s\text{ }r\text{ }s\text{ }r^3\text{ }s\text{ }r\text{ }s \, = \, 1$.}\\

\noindent In the following sections we utilize a proposition from \cite{Kappe}, by Kappe and Ratchford, that introduces properties of centralizer-like subgroups associated with the $n$-Engel word in groups. Before we introduce the proposition, definitions for conjugation, normal closure, and multiple sets need to be stated.
\begin{definition}
    Let $(g,h)$ and $(x,y)$ be elements of some group $G \times H$. Then $(g,h)$ \textbf{conjugated} by $(x,y)$ is denoted as $(g,h)^{(x,y)}$ and this can be expanded in the following way $$(g,h)^{(x,y)} = (x,y)^{-1}(g,h)(x,y).$$
\end{definition}
\begin{definition}
    The \textbf{normal closure} of $(g,h)$ in a group, $G \times H$, denoted $(g,h)^{G \times H}$, is the smallest normal subgroup that contains the set. Then $(g,h)^{G \times H}$ is defined as 
    \begin{align*}
        (g,h)^{G \times H} =& \{ (g,h)^{(x,y)} \, |\, (x,y) \in G \times H\} \\
        =& \{(x,y)^{-1}(g,h)(x,y) \, | \, (x,y) \in G \times H\}.
    \end{align*}
\end{definition}

\begin{definition}
    The \textbf{right centralizer-like subgroups associated with the 1-Engel word inside of a group $G \times H$, with an element $(g,h)$}, denoted $E_1^*(G \times H, (g,h))$, is defined as
     $$E_1^*(G \times H, (g,h)) = \{(a,b) \in G \times H \, | \, [(x,y)(a,b)\,,\,(g,h)] = [(x,y)\,,\,(g,h)] \text{ } \forall \,(x,y) \in G \times H\}.$$
     Notice \textbf{the left centralizer-like subgroup associated with the 1-Engel word inside of a group $G \times H$, with an element $(g,h)$}, denoted $^*E_1(G \times H, (g,h))$ is defined as \\
 $$^*E_1(G \times H, (g, h)) = \{(a,b) \in G \times H \text{ } | \text{ } [(a,b)(x,y)\,,\,(g,h)] = [(x,y)\,,\,(g,h)] \text{ } \forall \, (x,y) \in G \times H\}.$$
\end{definition}

\noindent 

\noindent The asterisk in the notation corresponds to the left or right side depending on which side the element $(a,b)$ is absorbed on. Thus each element in $G \times H$ will have both left and right centralizer-like subgroups associated with the 1-Engel word of that group. \newline
\noindent The two sets introduced in Definition 6 can be defined more generally without a specific element  $(g,h)$. If this were the case, the two set definitions would include a ``for all $(g,h) \in G \times H$" statement.

\begin{definition}
    The \textbf{set of right 1-Engel elements of a group $G \times H$, with an element $(g,h)$}, denoted $R_1(G \times H, (g,h))$ is defined as  $$R_1(G \times H,(g,h)) = \{(a,b) \in G \times H \text{ } | \text{ } [(a,b)\,,\,(g,h)] = (1,1)\}.$$
    In addition, the \textbf{set of right $n$-Engel elements of $G \times H$, with an element $(g,h)$} is defined by $$R_n(G \times H,(g,h)) = \{ (a,b) \in G \times H\text{ } | \text{ } [(a,b)\,,\,_n (g,h)] = (1,1)\}.$$
\end{definition}

\begin{definition}
    The \textbf{centralizer of an element $(g,h)$, in the group $G \times H$}, denoted $C_{G \times H}((g,h))$, is defined as
$$C_{G \times H}((g,h)) = \{(a,b) \in G \times H \text{ } | \text{ } [(a,b)\,,\,(g,h)] = (1,1)\}.$$
\end{definition}
\noindent Notice that the set of right 1-Engel elements of a group $G \times H$ with an element $(g,h)$ is equivalent to the centralizer of $(g,h)$ in $G \times H$.

\begin{definition}
    The \textbf{intersection of the set of right 1-Engel elements of a group $G \times H$, with the element $(g,h)$ conjugated by $(x,y)$}, denoted $\displaystyle \bigcap_{(x,y) \in G\times H}R_1(G \times H, (g,h)^{(x,y)})$, is defined as
\begin{align*}
    \displaystyle \bigcap_{(x,y) \in G \times H}R_1(G \times H, (g,h)^{(x,y)}) =& \{(a,b) \in G \times H \text{ } | \text{ } [(a,b)\,,\,(g,h)^{(x,y)}] = (1,1)\} \\
    =& \{(a,b) \in (g,h) \text{ } | \text{ } [(a,b)\,,\,(x,y)^{-1}(g,h)(x,y)] = (1,1)\}.
\end{align*}
\end{definition}

\begin{definition}
    \textbf{The centralizer of the normal closure of $(g,h)$ in the group $G \times H$}, denoted $C_{G \times H}((g,h)^{G \times H})$, is defined as 
$$C_{G \times H}((g,h)^{G \times H}) = \{(a,b) \in G \times H \text{ } | \text{ } [(a,b)\,,\,(g,h)^{G \times H}] = (1,1)\}$$ 
where $(g,h)^{G\times H}$ is the normal closure of $(g,h)$ in $G \times H$.
\end{definition}

\noindent The following two propositions inspired our results presented in this article. Proposition $1$ is by Kappe and Ratchford \cite{Kappe}, and Proposition $2$ is by McClellan and Tlachac \cite{McClellan}.
\begin{proposition}
    \cite{Kappe} Let $G$ be a group, $g \in G$ and $g^G$ the normal closure of $g$ in $G$. Then,
    \begin{enumerate}
        \item $^*E_1(G, g) = R_1(G, g) = C_G(g)$
        \item $E_1^*(G, g) = \displaystyle \bigcap_{x \in G}R_1(G, g^x) = C_G(g^G)$
        \item $E_1^*(G, g) \, \triangleleft \,G$ and $E_1^*(G, g) \subseteq \text{ }^*E_1(G, g)$
        \item $E_1^*(G, g) = R_1(G, g)$.
    \end{enumerate}
\end{proposition}

\begin{proposition}
\cite{McClellan} For groups $G$ and $H$, $g \in G$ and $h \in H$, we know the following are true:
\begin{enumerate}
    \item $R_2(G \times H) = R_2(G) \times R_2(H)$
    \item $R_2(G \times H, (g,h)) = R_2(G,g) \times R_2(H,h)$.
\end{enumerate}
\end{proposition}

\section{General Results}

Proposition $3$ allows the reader to write an $n$-variable word as a direct product of words.  This result motivated the proof of Theorem $4$.

\begin{proposition}
    For a group G with $x_1, \dots, x_i, x_{i+1}, \dots, x_n \in G$ we have $$w(x_1, \dots, x_i, x_{i+1}, \dots, x_n) = w(x_1, \dots, x_i) \times w(x_{i+1}, \dots, x_n)$$ where $w(x_1, \dots, x_n)$ denotes an $n$-variable word.
\end{proposition}
\begin{proof}
     Let $G$ be a group, let $x_1, . . . , x_n \in G$. Let $w(x_1, . . . , x_i, x_{i+1}, . . . , x_n)$ be an $n$-variable word in G.We want to show that $w(x_1, \dots, x_i, x_{i+1}, \dots, x_n) = w(x_1, \dots, x_i) \times w(x_{i+1}, \dots, x_n)$ for all $n \geq 1$ where $1 \leq i \leq n$. \newline
    \underline{Base Case:} Let $n=1$. Trivially $w(x_1)=w(x_1)$ and so the base case holds.\\
    \underline{Induction Case:} Assume $w(x_1, \dots, x_i, x_{i+1}, \dots, x_k) = w(x_1, \dots, x_i) \times w(x_{i+1}, \dots, x_k)$ for some $k$ such that \\$1 \leq i \leq k \leq n$. Notice that $w(x_1, \dots, x_i, x_{i+1}, \dots, x_k) = x_1 \cdots x_ix_{i+1} \cdots x_k$. In addition, \\$w(x_1, \dots, x_i) = x_1 \cdots x_i$ and $w(x_{i+1}, \dots, x_k) = x_{i+1} \cdots x_k$. Thus it follows that $$x_1 \cdots x_ix_{i+1} \cdots x_k = w(x_1, \dots , x_i,x_{i+1}, \dots, x_k)= w(x_1, \dots, x_i) \times w(x_{i+1}, \dots, x_k).$$ We now want to show that $w(x_1, \dots, x_i, x_{i+1}, \dots, x_k , x_{k+1}) = w(x_1, \dots, x_i) \times w(x_{i+1}, \dots, x_k, x_{k+1})$. Observe the following
\begin{align*}
    w(x_1, \dots, x_i, x_{i+1}, \dots, x_k , x_{k+1}) &= x_1 \cdots x_ix_{i+1} \cdots x_kx_{k+1}\\
    &= (x_1 \cdots x_ix_{i+1} \cdots x_k)x_{k+1}\\
    &= w(x_1, \dots, x_i, x_{i+1}, \dots, x_k) w(x_{k+1}) &&\text{by the inductive hypothesis}\\
    &= w(x_1, \dots, x_i) \times w(x_{i+1}, \dots, x_k) w(x_{k+1}) &&\text{by the inductive hypothesis}\\
    &= w(x_1, \dots, x_i) \times w(x_{i+1}, \dots, x_k, x_{k+1}).
\end{align*}
Thus $w(x_1, \dots, x_i, x_{i+1}, \dots, x_k, x_{k+1}) = w(x_1, \dots, x_i) \times w(x_{i+1}, \dots, x_k, x_{k+1})$. Therefore $$w(x_1, \dots, x_i, x_{i+1}, \dots, x_n) = w(x_1, \dots, x_i) \times w(x_{i+1}, \dots, x_n)$$ for all $n \geq 1$ where $1 \leq i \leq n$.
\end{proof}

\section{The Main Theorems}

We seek two main characterizations.  The first characterization is of right $n$-Engel elements of a direct product of groups with the direct product of the set of right $n$-Engel elements of each direct factor.  The second characterization provides two equivalences
involving centralizer-like subgroups associated with the $1$-Engel word, right $1$-Engel elements, centralizers, and centralizers of normal closures in a direct product of groups.  We begin with our first characterization.
\begin{theorem}
     Let $G, H$ be groups and consider $G \times H$. Let $(g,h) \in G \times H$. Then we know that the following are true
     \begin{enumerate}
         \item $R_n(G \times H) = R_n(G) \times R_n(H)$
         \item $R_n(G \times H, (g,h)) = R_n(G,g) \times R_n(H,h)$.
     \end{enumerate}
\end{theorem}
\begin{proof}
    Let $G, H$ be groups and consider $G \times H$. Let $(g,h) \in G \times H$.
    \begin{enumerate}
        \item Consider $R_n(G \times H)$ to be the set of right $n$-Engel elements of the direct product group and, $R_n(G)$ to be the set of right $n$-Engel elements of $G$, and $R_n(H)$ to be the set of right $n$-Engel elements of $H$. We want to show that $R_n(G \times H) = R_n(G) \times R_n(H)$ for all $n \geq 1$. We proceed by induction. \\
        \underline{Base Case:} This is an exercise left to the reader. The logic used for the base case is the same as that used in the induction case.\\
        \underline{Induction Case:} Assume for some integer $k \geq 1$ we have $R_k(G \times H) = R_k(G) \times R_k(H)$. The $k$-Engel words are made up of elements from their respective groups, because of this the nested commutators can be written as words. Then let $[a\,,\,_k u] = w(a, \dots, u)$ and let $[b\,,\,_k v] = w(b, \dots, v)$ where $w(a, \dots, u)$ represents the word generated by the nested commutator $[a\,,\,_k u]$ and $w(b, \dots, v)$ represents the word generated by the nested commutator $[b\,,\,_k v]$. Because we can write the nested commutators as words we can also write $$[(a,b)\,,\,_k (u,v)] = [a\,,\,_k u] \times [b\,,\,_k v] = ( w(a, . . . , u), w(b, . . . , v) ).$$
        Now we want to show $R_{k+1}(G \times H) = R_{k+1}(G) \times R_{k+1}(H)$ Let $(a,b) \in R_{k+1}(G \times H)$ such that $[(a,b)\,,\,_{k+1} (u,v)] = (1,1)$ for all $u \in G \text{ and for all }$ $v \in H$. Then 
        \begin{align*}
            (1,1) &= [(a,b)\,,\,_{k+1} (u,v)]\\
            &= [ [ (a,b)\,,\,_k (u,v) ]\,,\,(u,v)]\\
            &= [(w(a, \dots, u), w(b, \dots, v))\,,\,(u, v)] &&\text{by the inductive hypothesis}\\
            &= (w(a, \dots, u), w(b, \dots, v))^{-1} (u, v)^{-1} (w(a, \dots, u), w(b, \dots,v) (u, v)\\
            &= (w(a, \dots, u)^{-1}, w(b, \dots, v)^{-1}) (u^{-1}, v^{-1}) (w(a, \dots, u), w(b, \dots, v)) (u,v)\\
            &= (w(a, \dots, u)^{-1}u^{-1}w(a, \dots, u) u, w(b, \dots, v)^{-1} v^{-1} w(b, \dots, v) v)
        \end{align*}
         Thus we have $w(a, \dots, u)^{-1} u^{-1} w(a, \dots, u) u = 1 \text{ and } w(b, \dots, v)^{-1}v^{-1} w(b, \dots, v) v = 1$. Consider the first equation in the previous sentence
         \begin{align*}
             1 &= w(a, . . . , u)^{-1} u^{-1} w(a, . . . , u) u\\
             &= [w(a, . . . , u)\,,\,u]\\
             &= [[a\,,\,_k u] , u] &&\text{by the inductive hypothesis}\\
             &= [a\,,\,_{k+1} u].
         \end{align*}
         The same can be done for $w(b, \dots, v)^{-1}v^{-1} w(b, \dots, v) v = 1$ which implies that $[b\,,\,_{k+1}v]=1$. Thus $(a,b) \in R_{k+1}(G) \times R_{k+1}(H)$ and so $R_{k+1}(G \times H) \subseteq R_{k+1}(G) \times R_{k+1}(H)$. Thus it follows that $R_{k+1}(G \times H) = R_{k+1}(G) \times R_{k+1}(H)$. Therefore $R_n(G \times H) = R_n(G) \times R_n(H)$ for all $n \geq 1$.
         \item The proof of $R_n(G \times H, (g,h)) = R_n(G,g) \times R_n(H,h)$ is the same as the proof of property 1; the difference lies in the sets. Instead of proving that $R_n(G \times H) = R_n(G) \times R_n(H)$ for all elements in the respective groups, proving $R_n(G \times H, (g,h)) = R_n(G,g) \times R_n(H,h)$ is for one specific element in each group. Replacing $u$ with $g$, $v$ with $h$, and removing the for all elements statements in the conditions for the sets results in the proof of property 2.
    \end{enumerate}
\end{proof}
\noindent Our second characterization follows. The side the asterisk is on dictates whether the group is the left or right, respectively, for the centralizer-like subgroups associated with the $1$-Engel word.

\begin{theorem}
Let $G$ and H be groups where $(u, v) \in G \times H$ and $(u, v)^{G \times H}$ is the normal closure of $(u, v)$ in $G \times H$. Then,

\begin{enumerate}
    \item $^*E_1(G\times H, (u, v)) = R_1(G \times H, (u, v)) = C_{G\times H}((u, v))$
    \item $E_1^*(G \times H, (u, v)) = \displaystyle \bigcap_{(x, y) \text{ }\in \text{ }G \times H}R_1(G \times H, (u, v)^{(x, y)}) = C_{G\times H}((u, v)^{G\times H})$.
\end{enumerate}
\end{theorem}

\begin{proof}
    \noindent Let $G$ and $H$ be groups. Let the following definitions hold,

\begin{enumerate}
    \item To start, we consider the first equality. We begin by proving $^*E_1(G \times H, (u,v)) \subseteq R_1(G \times H, (u,v))$. Let $(a, b)\in \text{ }^*E_1(G \times H, (u, v))$. Then using the techniques stated in the introduction it follows that
\begin{align*}
    [(a, b)(x, y)\,,\,(u, v)]  =& [(x, y)\,,\,(u, v)]\\
    ((a, b)(x, y))^{-1}(u, v)^{-1}(a, b)(x, y)(u, v) =& (x, y)^{-1}(u, v)^{-1}(x, y)(u, v) \\
    (x^{-1}, y^{-1})(a^{-1}, b^{-1})(u^{-1}, v^{-1})(a, b)(x, y)(u, v) =& (x^{-1}, y^{-1})(u^{-1}, v^{-1})(x, y)(u, v)\\
    (x^{-1}a^{-1}u^{-1}axu, y^{-1}b^{-1}v^{-1}byv) =& (x^{-1}u^{-1}xu, y^{-1}v^{-1}yv)\\
    (x^{-1}a^{-1}u^{-1}ax, y^{-1}b^{-1}v^{-1}by) =& (x^{-1}u^{-1}x, y^{-1}v^{-1}y)\\
    (a^{-1}u^{-1}ax, b^{-1}v^{-1}by) =& (u^{-1}x, v^{-1}y)\\
    (a^{-1}u^{-1}a, b^{-1}v^{-1}b) =& (u^{-1}, v^{-1}) \\
    (a^{-1}u^{-1}au, b^{-1}v^{-1}bv) =& (1, 1)\\
    (a^{-1}, b^{-1})(u^{-1}, v^{-1})(a, b)(u, v) =& (1, 1)\\
    (a, b)^{-1}(u, v)^{-1}(a, b)(u, v) =& (1, 1)\\
    [(a, b)\,,\,(u, v)] =& (1, 1)
\end{align*}
\noindent Therefore $(a, b) \in R_1(G \times H, (u, v))$. Thus $^*E_1(G \times H, (u, v)) \subseteq R_1(G \times H, (u, v))$. To prove\\ $R_1(G \times H, (u, v)) \subseteq  {^*}E_1(G\times H, (u, v))$, the exact same steps are applied, just in reverse order. Therefore $^*E_1(G \times H, (u, v)) = R_1( G \times H, (u, v))$. Next, notice $R_1(G \times H, (u, v))$ and $C_{G \times H}((u, v))$ are equivalent because they have the same set definition. Finally, through the transitive property, since $^*E_1(G \times H, (u, v)) = R_1(G \times H, (u, v))$ and $R_1(G \times H, (u, v)) = C_{G \times H}((u, v))$, then $^*E_1(G \times H, (u, v)) = C_{G \times H}((u, v))$. 
\item The proof of $E_1^*(G \times H, (u, v)) = \displaystyle \bigcap_{(x, y) \text{ }\in \text{ }G \times H}R_1(G \times H, (u, v)^{(x, y)}) = C_{G\times H}((u, v)^{G\times H})$ can be proved analogously to the proof above. In property 1 only one set containment needs to be proved as two of the sets have the same set definition. However, in property 2 two of the set containment's need to be proved in order to show the property holds generally. Exercise is left to reader.
\end{enumerate}
\end{proof}

\section{Continuing Research}

\noindent The second half of our research project was spent trying to find counterexamples for Conjecture $6$ and trying to prove them generally. Our attempts for proving these conjectures are detailed in this section along with an example of a specific direct product of groups for which our conjectures held.
\begin{conjecture}
Let $G \times H$ be a group, let $(u,v) \in G \times H$, and $(u,v)^{G \times H}$ the normal closure of $(u,v) \in G \times H$. Then
\begin{enumerate}
    \item $E_1^*(G \times H,(u,v)) = R_1(G \times H,(u,v))$.
    \item $E_1^*(G \times H, (u,v)) \triangleleft G \times H$ and $E_1^*(G \times H,(u,v)) \subseteq \,^*E_1(G \times H,(u,v))$
\end{enumerate}
\end{conjecture}

\subsection{Attempted Proofs of Conjecture 6}
\subsubsection{Property $1$}

Let $G \times H$ be a group and let $(u,v) \in G \times H$. Let $(a,b) \in R_1(G \times H, (u,v))$. Then we know that $[(a,b)\,,\,(u,v)] = (1,1)$. To show that $(a,b) \in E_1^*(G \times H,(u,v))$, we must show that \\ $[(x,y)(a,b)\,,\,(u,v)] = [(x,y)\,,\,(u,v)]$ for all $(x,y) \in G \times H$. Note that $(x,y) \in R_1(G \times H, (u,v))$ or $(x,y) \not\in R_1(G \times H, (u,v))$. We proceed by cases. \newline

\noindent \underline{Case 1:} Suppose that $(x,y) \in R_1(G \times H, (u,v))$. Then $[(x,y)\,,\,(u,v)] = (1,1)$. We know that \newline $R_1(G \times H, (u,v)) \leq G \times H$. Then $R_1(G \times H, (u,v))$ is closed under the binary operation of $G \times H$ and is closed under inverses. This means that for every $(x,y) \in G \times H$ where $(x,y) \in R_1(G \times H, (u,v))$, we have $(x,y)(a,b) \in R_1(G \times H, (u,v))$. Then since $(x,y)(a,b) \in R_1(G \times H, (u,v))$, it follows that \newline $[(x,y)(a,b)\,,\,(u,v)] = (1,1)$. Since $[(x,y)(a,b)\,,\,(u,v)] = (1,1)$ and $[(x,y)\,,\,(u,v)] = (1,1)$, we can conclude that $[(x,y)(a,b)\,,\,(u,v)] = [(x,y)\,,\,(u,v)]$. Therefore, $(a,b) \in E_1^*(G \times H,(u,v))$ when $(x,y) \in R_1(G \times H, (u,v))$. \newline

\noindent \underline{Case 2:} Suppose that $(x,y) \not\in R_1(G \times H, (u,v))$. We understand that this case would contain a direct proof of the claim if it were to be true, but the results of this case were unexpected. We started by considering $ [(x,y)(a,b)\,,\,(u,v)]$ and we attempted to manipulate it to equal $[(x,y)\,,\,(u,v)]$. Instead we found several different results with one of the attempts outlined below. Note that all of the results were obtained through expansion, multiplying on the left and right by inverses and inserting identity's within the commutator. Through manipulation we found $ [(x,y)(a,b)\,,\,(u,v)]= (a,b)^{-1}[(x,y)\,,\,(u,v)](a,b)$. This case had the commutator we wanted conjugated by $(a,b)$. Since $(a,b) \in R_1(G \times H, (u,v))$, then $[(a,b)\,,\,(u,v)] = (1,1)$. We attempted to start with $[(a,b)\,,\,(u,v)] = (1,1)$ and show $[(a,b)\,,\,(u,v)] = [(x,y)\,,\,(u,v)]$. By inserting the identity and multiplying on the left and right by inverses, we manipulated the commutators but were unable to show the equality we wanted.\\ Since we were unable to prove $R_1(G\times H, (u,v)) \subseteq E_1^*(G \times H, (u,v))$ generally, we attempted to prove the other containment, i.e., $E_1^*(G \times H, (u,v)) \subseteq R_1(G\times H, (u,v))$. We hoped in proving the other containment it would illustrate a technique to use in our first containment to show it holds generally. Let $G \times H$ be a group with $(u,v) \in G \times H$. Consider $(a,b) \in E^*_1(G \times H, (u,v))$. We started with $[(x,y)(a,b)\,,\,(u,v)] = [(x,y)\,,\,(u,v)]$ and wanted to show $[(x,y)(a,b)\,,\,(u,v)] = [(a,b)\,,\,(u,v)]$. Through two attempts, one by multiplying on the right by inverse elements and the other by multiplying on the left by inverse elements, we altered the commutator but were unable to show the equality. In addition, we were unable to find any helpful techniques to use in our initial containment proof. Thus we were unable to prove that $R_1(G\times H, (u,v)) = E_1^*(G \times H, (u,v))$.

\subsubsection{Property $2$}

It was verified that $E_1^*(G \times H,(u,v))$ is indeed a subgroup of $G \times H$. Thus to prove $E_1^*(G \times H,(u,v))$ is normal in $G \times H$, it suffices to show that $(g,h) \, E_1^*(G\times H, (u, v)) \, (g,h)^{-1} \subseteq E_1^*(G\times H, (u, v))$ for all $(g,h) \in G \times H$. Let $(a,b) \in E_1^*(G\times H, (u, v))$. We considered $[(x,y)(g,h)^{-1}(a,b)(g,h)\,,\, (u,v)]$ and attempted to show it equal to $[(x,y)\,,\,(u,v)]$. Within our expansion and manipulation of the commutator we found that $ [(x,y)(g,h)^{-1}(a,b)(g,h)\,,\, (u,v)]=((a,b)^{-1})^{(g,h)}((x,y)^{-1}(u,v)^{-1}(x,y))(a,b)^{(g,h)}(u,v)$. Notice in the last line how the commutator that we want is separated by $(a,b)^{(g,h)}$ and that $((a,b)^{-1})^{(g,h)}$ is on the left of the commutator. In addition, notice how we can rewrite the last line in the following way
\begin{align*}
    [(x,y)(g,h)^{-1}(a,b)(g,h)\,,\, (u,v)] &= ((a,b)^{-1})^{(g,h)}((x,y)^{-1}(u,v)^{-1}(x,y))(a,b)^{(g,h)}(u,v) \\
    &= ((a,b)^{-1})^{(g,h)}((u,v)^{-1})^{(x,y)}(a,b)^{(g,h)}(u,v).
\end{align*}
Observe that by definition $((u,v)^{-1})^{(x,y)} \in C_{G \times H}((u, v)^{G \times H})$. Thus it suffices to show that $((u,v)^{-1})^{(x,y)}$ and $((a,b)^{-1})^{(g,h)}$ commute or to show that $((u,v)^{-1})^{(x,y)}$ and $(a,b)^{(g,h)}$ commute. This is because if these elements commute, $((a,b)^{-1})^{(g,h)}$ and $(a,b)^{(g,h)}$ would cancel, resulting in $[(x,y),(u,v)]$ as desired. We attempted to show this held but were unable to prove commutativity in either case. Therefore we were unable to prove $E_1^*(G\times H, (u, v)) \triangleleft G \times H$ generally. \\

\noindent Because we were unsuccessful in proving our conjectures generally, we decided to find a direct product group for which the conjecture held. The following section discusses a specific group and its properties which we believe have effect on the conjectures being true.

\subsubsection{Conjecture 6 Example} 

Conjecture $6$ held for $K_4 \times S_3$, where $K_4$ is the Klein four-group and $S_3$ is the symmetric group of order six. We examined $K_4 \times S_3$ to see if any of the group properties provided illumination as to why the conjecture held for the example but not for the general direct product.\\

\noindent We proved $R_1(K_4 \times S_3, (c,(1\,3\,2)) = E_1^*(K_4 \times S_3, (c,(1\,3\,2))$ and $E_1^*(K_4 \times S_3, (c,(1\,3\,2))) \triangleleft K_4 \times S_3$.  The sets $R_1(K_4 \times S_3, (c,(1\,3\,2))$ and $E_1^*(K_4 \times S_3, (c,(1\,3\,2))$ were calculated using their respective set definitions.  We concluded that both $R_1(K_4 \times S_3, (c,(1\,3\,2))$ and $E_1^*(K_4 \times S_3, (c,(1\,3\,2))$ are equal to the set consisting of elements $(e,e),\,(e,(1\,2\,3\,)),\, (e,(1\,3\,2)),\, (a,e),\, (a,(1\,2\,3)),\, (a,(1\,3\,2)),\, (b,e),\, (b, (1\,2\,3)),\, (b,(1\,3\,2)),\, (c,e),\, \\ (c,(1\,2\,3)),\, (c,(1\,3\,2)) .$ \\ 

\noindent Now to show that $E_1^*(K_4 \times S_3, (c,(1\,3\,2))) \triangleleft K_4 \times S_3$ we used the $12$ elements listed in the previous paragraph.  To verify that $E_1^*(K_4 \times S_3, (c,(1\,3\,2)))$ is a normal subgroup of $K_4 \times S_3$ we showed that for each $(x,y) \in E_1^*(K_4 \times S_3, (c,(1\,3\,2)))$ we have $(a,b)^{-1}(x,y)(a,b) \in E_1^*(K_4 \times S_3, (c,(1\,3\,2)))$ for all $(a,b) \in K_4 \times S_3$. Therefore the two conjectures $R_1(G\times H, (u,v)) = E_1^*(G \times H, (u,v))$ and $E_1^*(G \times H,(u,v)) \triangleleft G \times H$ held for $K_4 \times S_3$ with the element $(c,(1\,3\,2))$.\\

\noindent Observe that both $K_4$ and $S_3$ are a non metabelian and solvable groups. We were able to show that $K_4 \times S_3$ is a non metabelian, solvable group. For more information on metabelian and solvable groups, see \cite{Doerk} and \cite{Huppert}.\\

\noindent To justify why $K_4 \times S_3$ is a non metabelian group it suffices to show that there exists a normal subgroup, $H$, of $K_4 \times S_3$ such that $(K_4 \times S_3)/H$ is nonabelian. All possible subgroups of $K_4 \times S_3$ were determined and from that extensive list, the normal subgroups of $K_4 \times S_3$ were found. Consider the normal subgroup $H = \{(e,e),\, (a,e),\, (b,e),\, (c,e)\}$. Observe that we have the following
\begin{align*}
     (K_4 \times S_3)/H =& \{(e,e)\cdot \{(e,e),\, (a,e),\, (b,e),\, (c,e)\}, ..., (c,(1\,3\,2))\cdot \{(e,e),\, (a,e),\, (b,e),\, (c,e)\}\} \\
     =& \{\{(e,e),\, (a,e),\, (b,e),\, (c,e)\},\, \{(e,(1\,2)),\, (a,(1\,2)),\, (b,(1\,2)),\, (c,(1\,2))\}, \\
     &\, \{(e,(2\,3)),\, (a,(1\,2)),\,(b,(2\,3)),\, (c,(2\,3))\}, \{(e,(1\,3)),\, (a,(1\,3)),\, (b,(1\,3)),\, (c,(1\,3))\}, \\
     &\{(e,(1\,2\,3\,)),\, (a,(123)),\, (b, (1\,2\,3)),\, (c,(1\,2\,3))\}, \{(e,(1\,3\,2)),\, (a,(1\,3\,2)),\, (b,(1\,3\,2)),\, (c,(1\,3\,2))\}\} \\
     =& \{H,\, (e,(1\,2))H,\, (e,(2\,3))H,\, (e,(1\,3))H,\, (e,(1\,2\,3\,))H,\, (e,(1\,3\,2))H\}.
 \end{align*}

\noindent We want to show $[aH, bH]=H $ for all $aH, bH \in (K_4 \times S_3)/H$ in order to show that $(K_4 \times S_3)/H$ is abelian. It can be computed that $[(e,(1\,2))H\,,\,(e,(1\,3\,2))H] = (e,(1\,2\,3\,))H$. Because $(e,(1\,2\,3\,))H \neq H$ it follows that $(K_4 \times S_3)/H$ is nonabelian. Hence $K_4 \times S_3$ is a non metabelian group.\\

\noindent To justify why $K_4 \times S_3$ is a solvable group it suffices to show that there exists a chain of normal subgroups such that the condition of abelian quotient groups is met. For the sake of brevity, we will discuss only one chain of normal subgroups, but we found three chains. It can be verified that we have $$\{(e,e)\} \triangleleft \langle (a,e) \rangle \triangleleft \langle (a,(1\,3\,2))\rangle \triangleleft R_1(K_4 \times S_3, (c, (1\,3\,2))) \triangleleft K_4 \times S_3.$$ Now that we have a chain of normal subgroups, we must check for abelian quotient groups. \newline
Through computations it can be shown that
\begin{align*}
    (R_1(K_4 \times S_3, (c, (1\,3\,2))))/\langle (a, (1\,3\,2)) \rangle &= \{ \langle (a, (1\,3\,2)) \rangle, \, (b, (1\,2\,3))\langle (a, (1\,3\,2)) \rangle\}.
\end{align*}
It suffices to check that $\langle (a, (1\,3\,2)) \rangle$ and $(b, (1\,2\,3))\langle (a, (1\,3\,2)) \rangle$ commute. Then we calculate \\$[\langle (a, (1\,3\,2)) \rangle\,,\, (b, (1\,2\,3))\langle (a, (1\,3\,2)) \rangle]$ to be $\langle (a, (1\,3\,2)) \rangle$ which proves commutativity. Therefore \\$(R_1(K_4 \times S_3, (c, (1\,3\,2))))/\langle (a, (1\,3\,2)) \rangle$ is abelian. Similar work was done to show the other quotient groups are abelian. Hence $K_4 \times S_3$ is solvable. \\

\noindent After examination of the group properties of $K_4 \times S_3$, we conclude solvability of a direct product of groups may aid in the proving of the conjectures generally.  For the future we suggest attempting using non metabelian and solvable direct products.

\section{Conclusion}

In this article we provided two characterizations, Theorem $4$ and Theorem $5$, of centralizer-like subgroups associated with $n$-Engel words in a direct product of groups. Future topics to explore include prove the isomorphism theorems for centralizer-like subgroups associated with the $n$-Engel word in a direct product of groups and/or semi-direct products, finding a counterexample for Conjecture 6, and characterizing centralizer-like subgroups associated with $n$-Engel words inside a semi-direct product of groups.\\

\noindent This project was supported by the NSF DMS-1451801 during the Summer REU at the University of Wisconsin-Eau Claire (UWEC). Special thanks to Dr. Christopher Davis at UWEC, for his helpful insight regarding the direct product of words.

\end{document}